\documentclass[12pt]{article}
\usepackage{amsfonts}
\usepackage{eurosym}
\usepackage{amssymb}
\usepackage{amsmath,amssymb,latexsym}
\usepackage[dvips]{graphics}
\newtheorem{theorem}{Theorem}[section]

\newtheorem{definition}[theorem]{Definition}

\newtheorem{corollary}[theorem]{Corollary}
\newtheorem{remark}[theorem]{Remark}

\newtheorem{proof}[theorem]{Proof}
\numberwithin{equation}{section}

\begin{document}

\title{\textbf{Geometry of left invariant Randers metric on the Heisenberg group}}
\author{{Ghodratallah Fasihi Ramandi$\sp{a}$\thanks{fasihi@sci.ikiu.ac.ir (Gh. Fasihi-Ramandi)}\qquad
{Shahroud Azami}$\sp{a}$\thanks{azami@sci.ikiu.ac.ir(S. Azami)}}\\
$\sp{a}${\small \textit{Department of Mathematics, Faculty of Science, Imam Khomeini International University, Qazvin, Iran}}}
%EndAName
\date{}
\maketitle
\begin{abstract}
In this paper, we investigate the geometry of left-invariant Randers metrics on the Heisenberg group $H_{2n+1}$, of dimension $2n+1$. Considering a left-invariant Randers metric, we give the Levi-Civita connection, curvature tensor, Ricci tensor and scalar curvature and show  the Heisenberg groups $H_{2n+1}$  have constant negative scalar curvature. Also, we show the Heisenberg group $H_{2n+1}$ can not admit Randers metric of Berwald and Ricci-quadratic Douglas types. Finally, an explicit formula for computing flag curvature is obtained which shows that there exist flags of strictly negative and strictly positive curvatures. 
\end{abstract}
%\keywords{Heisenberg group, Randers metrics, Flag curvature.}

\hspace{-1cm}\rule{\textwidth}{0.2mm}
\section{Introduction}
Research into left-invariant Riemannian metrics on Lie groups is an active subject of research and this topic is mentioned among
many author's works so far. The curvature properties of such metrics on various kinds of Lie groups are mainly investigated in classical works of Milnor (see \cite{Milnor}). \\
Randers metrics as a special case of Finsler metrics are constructed by Reimannian metrics and vector fields. Similar to the Riemannian case, the notion of left-invariant Randers metrics on a Lie group $G$ is defined, and the geometry of such spaces is part of many author's interest topic. A general study of Berwald-type Randers metric on two-step homogeneous nilmanifolds of dimension five is done in \cite{salimi}. Also, curvature properties of Douglas-type Randers metrics on five dimensional two-step homogeneous nilmanifolds can be found in \cite{Nasehi}.\\

The Heisenberg groups play a crucial role in theoretical physics, and they are well understood from the viewpoint of sub-Riemannain geometry. These groups arise in the description of one-dimensional quantum mechanical systems. More generally, one can consider Heisenberg groups associated to $n$-dimensional systems, and most generally, to any symplectic vector space. \\

In this study, we develop the results of \cite{Kovacs} for a special case of five-dimensional Heisenberg group by investigating the geometry of left-invariant Randers metrics on the Heisenberg group $H_{2n+1}$, of dimension $2n+1$. Considering a left-invariant Randers metric, we give the Levi-Civita connection, curvature tensor, Ricci tensor and scalar curvature and show  the Heisenberg groups $H_{2n+1}$  have constant negative scalar curvature. Also, we show the Heisenberg group $H_{2n+1}$ can not admit Randers metric of Berwald and Ricci-quadratic Douglas types. Finally, an explicit formula for computing flag curvature of $Z$-Randers metrics is obtained.
%%%%%%%%%%%%%%%%%%%%%%%%%%%%%%%%%%%%%%%%%%%%%%%
\section{Preliminaries}
In this section, we summarize the main concepts and definitions that are needed in this paper.
\begin{definition}\label{inv}
A (semi-)Riemannian metric $g$ on a Lie group $G$  is said to be left-invariant if $L_a^* (g)=g$, for all $a\in G$.
\end{definition}

It is well-known that there is a bijective correspondence between left-invariant metrics on a Lie group $G$, and inner products on its associated Lie algebra $\mathfrak{g}=T_e G$. So, the geometry of a left-invariant metric on a Lie group $G$ can be recovered from the geometry of its associated inner product space $\mathfrak{g}=T_e G$. For instance, let $G$ be a Lie group with a left-invariant metric $g$ then the Koszul formula on $U,V,W\in \mathfrak{g}$ is given by
\[2\langle \nabla_U V, W\rangle=\langle[U,V],W\rangle-\langle[V,W],U\rangle+\langle[W,U],V\rangle,\]
where, $\langle.,.\rangle$ denotes the induced inner product on $\mathfrak{g}$ by $g$. \\

The rest of this section is devoted to remind some basic notions on Finsler geometry and in particular to developing the  definition \ref{inv} for Finsler manifolds.

A Finsler metric on a manifold $M$ is a function $F:TM\to [0,\infty)$ with the following properties:
(i){\bf Regularity:} $F$ is smooth on the entire slit tangent bundle $TM\backslash \{0\}$.\\
(ii){\bf Positive homogeneity:} $F(x,\lambda y)=\lambda F(x,y)$ for all $\lambda>0$.\\
(iii){\bf Strong convexity:} The $n\times n$ Hessian matrix
\[ [g_{ij}]=[\dfrac{1}{2}\dfrac{\partial^2 F^2}{\partial y^i \partial y^j}] \]
is positive definite at every point $(x,y)\in TM\backslash \{0\}$.\\

The Finsler geometry counterpart of the Riemannian sectional curvature is the notion of flag curvature. The flag curvature is defined as follows:
\[K(P,Y)=\dfrac{g_Y (R(X,Y)Y,X)}{g_Y (Y,Y)g_Y (X,X) -g_Y^2 (Y,X)}\]
where,
\[g_Y(X,Z)=\dfrac{1}{2}\dfrac{\partial^2}{\partial s \partial t}|_{s=t=0} F^2(Y+sX+tZ),\] 
is the osculating Riemannain metric, $P=span \{X,Y\}$, $R(X,Y)(Z)=\nabla_X \nabla_Y Z -\nabla_Y \nabla_X Z-\nabla_{[X,Y]} Z$ and $\nabla$ is the Chern-Rund connection induced by $F$ on the pull-back bundle $\pi^* TM$ (see \cite{Bao}).\\

A special case of Finsler metrics are Randers metrics which are constructed by Riemannaian metrics and vector fields (1-forms). In fact, for a Riemannian metric $g$ and a vector field $X$ on $M$ such that $\sqrt{g(X,X)}<1$, the Randers metric $F$, defined by $g$ and $X$, is a Finsler metric as follows:
\begin{equation}\label{randers}
F(x,y)=\sqrt{g(y,y)}+g(X(x),y) \qquad \forall x\in M,\quad y\in T_x M.
\end{equation}

A Randers metric of the form (\ref{randers}) is called Berwald type if and only if the vector field $X$ is parallel with respect to Levi-Civita connection of $g$. It is well-known that in the such metrics the Chern connection of Randers metric $F$ coincide with the Levi-civita connection of $g$. Also, if for a Randers metric of the form (\ref{randers}) the 1-form $g(X,.)$ is closed, then the Randers metric is said to be of Douglas type. In the case that $M=G$ is a Lie group, one can easily check that a Randers metric on $G$ with underlying left-invariant Riemannian metric is of Douglas type if and only if its underlying vector field $Q$ satisfies the following equation
\[<[U,V],Q>=0,\qquad \forall U,V\in \mathfrak{g}=T_eG.\]

Now, we are prepare to generalize the definition \ref{inv} for Finsler manifolds.
\begin{definition}
A Finsler metric $F$ on a Lie group $G$ is said to be left-invariant if
\[F(a,v)=F(e,L_{a^{-1}}^* (v))\qquad \forall a\in G,\quad v\in T_a G,\]
where, $e$ is the unit element of $G$.
\end{definition}
Suppose that $G$ is a Lie group, $g$ and $U$ are left-invariant Riemnnian metric and left-invariant vector field on $G$, respectively, such that $\sqrt{g(U,U)}<1$. Then the formula (\ref{randers}) defines a left-invariant Randers metric on $G$. In fact, one can easily check that there is a bijective correspondence between left-invariant Randers metrics on a Lie group $G$ with the underlying Reimannian metric $g$ and the left-invariant-vector fields with length $<1$. Therefore, the invariant Randers metrics are one-to-one corresponding to the set (see \cite{salimi}, Proposition 3.1)
\[\{ U\in \mathfrak{g}| <U,U><1 \}. \]

Suppose that $G$ is a Lie group with a left-invariant Randers metric $F$ which is defined by a $U\in \mathfrak{g}$, then the $S$-curvature is given by
\[ S(e,v)=\dfrac{n+1}{2}\Big(  {\dfrac{\langle {U,v},\langle v,U\rangle U-v\rangle}{F(e,v)}}-\langle [U,v],U\rangle \Big),
\]
and the Randers metric has vanishing $S$-curvature if and only if the linear endomorphism $ad(U)$ of $\mathfrak{g}$ is skew symmetric with respect to the inner product $\langle . ,. \rangle$ on the Lie algebra $\mathfrak{g}$. (See \cite{Deng} Proposition 7.5).
\section{Geometry of Heisenberg groups}
In this section, we investigate the Riemannian geometry of left-invariant metrics on the Heisenberg group $H_{2n+1}$, of dimension $2n+1$.\\

The Heisenberg group $H_{2n+1}$ is defined on the base manifold $\mathbb{R}^{2n} \times \mathbb{R}$ by multiplication
\[(x,\lambda).(y,\mu)=(x+y,\lambda +\mu +\omega (x,y)),\]
where, $\omega$ denotes the standard symplectic form on $\mathbb{R}^{2n}$. Its associated Lie algebra $\mathbb{H}_{2n+1}$ is
\[\mathbb{H}_{2n+1}=\mathbb{R}^{2n} \oplus \mathbb{R}=\{(x,\lambda)| x\in \mathbb{R}^{2n},\lambda\in \mathbb{R}\},\]
with the following Lie bracket
\begin{equation}\label{bracket}
[(x,\lambda),(y,\mu)]=(0,\omega(x,y)).
\end{equation} 
The center of $\mathbb{H}_{2n+1}$ is one-dimensional, hence the Heisenberg group is 2-step nilpotent. In the other hand, every 2-step nilpotent Lie group of odd dimension with a one-dimensional center is locally isomorphism to the Heisenberg group $H_{2n+1}$.\\

Following \cite{vukmir}, any positive definite inner product on $\mathbb{H}_{2n+1}$ is given by the following theorem. For $\sigma_1 \geq \cdots \sigma_n \geq 1 $ and $\sigma=(\sigma_1,\cdots ,\sigma_n)$ denote 
\[D_n(\sigma)=\mathrm{diag} (\sigma_1,\sigma_1,\cdots ,\sigma_n ,\sigma_n)\]
\begin{theorem}\cite{vukmir}\label{diag}
Any positive definite inner product on $\mathbb{H}_{2n+1}$, up to the automorphism of Lie algebra (i.e., in some basis of $\mathbb{H}_{2n+1}$ such that the commutators are given by (\ref{bracket}) is represented by the matrix
\[
 \begin{pmatrix}
  D_{n-1}(\sigma) & 0 \\
   0 & S
 \end{pmatrix}\]
where $S=\mathrm{diag}(1,1,\lambda)$, $\lambda>0$.
\end{theorem}
We have already observed above that every positive definite inner product on the Heisenberg algebra $\mathbb{H}_{2n+1}$ with the commutator specified in (\ref{bracket}) has a diagonal representation. According to theorem \ref{diag}, we calculate the geometry of a left-invariant metric on the Heisenberg group $H_{2n+1}$.\\

Let $\beta=\{u_1,v_1,\cdots ,u_n,v_n,z\}$ is a basis of $\mathbb{H}_{2n+1}$ with non-zero commutators
\[ [u_i,v_i]=z,\qquad i=1,\cdots, n, \]
which is a special form of the Lie bracket given by (\ref{bracket}). Suppose that 
\[ U_1,V_1,\cdots ,U_n,V_n, Z \]
denote the corresponding left-invariant vector fields on $H_{2n+1}$. Fix a Riemannain metric $g$ on the Heisenberg group $H_{2n+1}$ and denote $\sigma_n=1$. Then the Levi-Civita connection of $g$ is given by the following theorem.
\begin{theorem}\label{levi}
The Levi-Civita connection $\nabla$ of the left-invariant metric $g$, satisfies the following equations.
\begin{align*}
&\nabla_{U_i} V_i =\dfrac{1}{2}Z,\qquad\qquad \nabla_{V_i} U_i=-\dfrac{1}{2}Z,\\
&\nabla_{U_i} Z =-\dfrac{\lambda}{2\sigma_i}V_i,\qquad\quad \nabla_Z {U_i} =-\dfrac{\lambda}{2\sigma_i}V_i,\\
&\nabla_{V_i} Z =\dfrac{\lambda}{2\sigma_i}U_i,\qquad\qquad \nabla_Z {V_i} =\dfrac{\lambda}{2\sigma_i}U_i.
\end{align*}
\end{theorem}
\begin{proof}
Straightforward computations using Koszul's formula, show the above relations, for example we compute the first formula. We have,
\begin{align*}
&2\langle \nabla_{U_i}V_i ,U_j \rangle =\langle [U_i, V_i],   U_j \rangle -\langle [V_i,U_j],   U_i\rangle  +\langle [U_j ,U_i],  V_i \rangle =0,\\
&2\langle \nabla_{U_i}V_i , V_j\rangle =\langle [U_i, V_i],   V_j \rangle -\langle [V_i,V_j],   U_i\rangle  +\langle [V_j ,U_i],  V_i \rangle =0,\\
&2\langle \nabla_{U_i}V_i , Z\rangle =\langle [U_i, V_i],   Z \rangle -\langle [V_i, Z],   U_i\rangle  +\langle [Z, U_i ],  V_i \rangle =\langle Z,Z  \rangle =\lambda .
\end{align*}
Hence, $\nabla_{U_i} V_i =\dfrac{1}{2}Z$.
\end{proof}
\begin{theorem}
The Riemannian curvature tensor of $\nabla$, denoted by $R$, satisfies the following relations.
\begin{align*}
&R(U_i ,U_k)U_k =0, \qquad R(U_i ,V_k)V_k =-\dfrac{3\lambda}{4\sigma_k}\delta_{ik}U_k ,\qquad R(U_i ,Z)Z=\dfrac{\lambda^2}{4\sigma_i^2}U_i,\\
&R(V_i ,U_k)U_k =-\dfrac{3\lambda}{4\sigma_k}\delta_{ik}V_k ,\qquad  R(V_i ,V_k)V_k =0,\qquad R(V_i ,Z)Z=\dfrac{\lambda^2}{4\sigma_i^2}V_i,\\
&R(Z ,U_k)U_k =\dfrac{\lambda}{4\sigma_k}Z,\qquad R(Z ,V_k)V_k =\dfrac{\lambda}{4\sigma_k}Z,
\end{align*}
where, $\delta_{ij}$ is the Kronecker delta.
\end{theorem}
\begin{proof}
Routine computations show these results. For example we compute the second formula. One can write,
\begin{align*}
R(U_i ,V_k)V_k &=\nabla_{U_i}\nabla_{V_k} V_k -\nabla_{V_k}\nabla_{U_i} V_k -\nabla_{[U_i,V_k]} V_k\\
&=-\delta_{ik}(\dfrac{1}{2}\nabla_{V_k} Z +\nabla_Z V_k)=-\dfrac{3\lambda}{4\sigma_k}\delta_{ik}U_k.
\end{align*}
\end{proof}
\begin{theorem}
The Ricci curvature tensor of $\nabla$, denoted by $\mathrm{Ric}$, satisfies the following relations.
\begin{align*}
& \mathrm{Ric} (U_i ,U_j)=-\dfrac{\lambda \delta_{ij}}{2\sigma_i},\\
& \mathrm{Ric} (V_i ,V_j)=-\dfrac{\lambda \delta_{ij}}{2\sigma_i},\\
& \mathrm{Ric} (Z,Z)=\dfrac{\lambda^2}{2}\sum_{k=1}^n \dfrac{1}{\sigma_k^2}.
\end{align*}
\end{theorem}
\begin{proof}
It is straightforward to check these results. For instance, we can write:
\begin{align*}
\mathrm{Ric} (V_i ,V_j)&=\dfrac{1}{\lambda}\langle R(V_i ,Z)Z,V_j \rangle +\sum_{k=1}^n \dfrac{1}{\sigma_k} \langle R(V_i ,U_k)U_k , V_j \rangle \\
&= \dfrac{1}{\lambda}\langle \dfrac{\lambda^2}{4\sigma_i^2}V_i,V_j \rangle +\sum_{k=1}^n \dfrac{1}{\sigma_k} \langle -\dfrac{3\lambda}{4\sigma_k}\delta_{ik}V_k , V_j \rangle \\
&=  \dfrac{\lambda}{4\sigma_i^2}\langle V_i,V_j \rangle  -\dfrac{3\lambda}{4\sigma_i^2}\langle V_i , V_j \rangle \\
&=-\dfrac{\lambda}{2\sigma_i^2}\langle V_i , V_j \rangle =-\dfrac{\lambda}{2\sigma_i^2}\delta_{ij}\sigma_i =-\dfrac{\lambda\delta_{ij}}{2\sigma_i}.
\end{align*}
\\
{\bf Remark:} The above computations indicate that $H_{2n+1}$ neither cannot be Einstein nor the metric is conformally flat. The next theorem shows that the Heisenberg groups $H_{2n+1}$ have constant negative scalar curvature.
\end{proof}
\begin{theorem}
The scalar curvature of $\nabla$, denoted by $\mathrm{R}$, is given by the following formula.
\[\mathrm{R}=-\dfrac{\lambda}{2}\sum_{i=1}^n \dfrac{1}{\sigma^2}.
\]
\end{theorem}
\begin{proof} We have,
\begin{align*}
\mathrm{R}&=  \dfrac{1}{\lambda}Ric  (Z,Z) +\sum_{i=1}^n \dfrac{1}{\sigma_i} Ric(U_i ,U_i) +\sum_{i=1}^n \dfrac{1}{\sigma_i} Ric(V_i ,V_i)\\
&=\sum_{i=1}^n \dfrac{\lambda}{2} \dfrac{1}{\sigma_i^2}-\dfrac{\lambda}{2} \dfrac{1}{\sigma_i^2}-\dfrac{\lambda}{2} \dfrac{1}{\sigma_i^2}\\
&=-\dfrac{\lambda}{2}\sum_{i=1}^n \dfrac{1}{\sigma^2}.
\end{align*}
\end{proof}
\section{Main Results}
In this section, our main results will be stated.
\begin{theorem}\label{berwa}
There is not any left-invariant Randers metric of Berwald type on the Heisenberg group $\mathbb{H}_{2n+1}$. 
\end{theorem}
\begin{proof}
As mentioned before, a left invariant Randers metric on Lie group $G$ is constructed by a left-invariant vector field with length $<1$. In the other hand, by definition a Randers metric is Berwald type if and only if the vector field be parallel with respect to the Levi-Civita connection. Suppose that $Q$ is a left-invariant vector field on $\mathbb{H}_{2n+1}$ which is parallel with respect to the Levi-Civita connection. This implies that for all $U\in \mathbb{H}_{2n+1}$ we have $\nabla_U Q=0$. Let 
\[Q=aZ+\sum_{i=1}^n a_i U_i+ b_i V_i,
\]
then the relations in theorem \ref{levi} show that $a=a_i =b_i =0$ for $1\leq i\leq n$, which is a contradiction.
\end{proof}
Note that a Finsler metric is said to be Ricci-quadratic if its Ricci curvature $\mathrm{Ric}(x,y)$ is quadratic with respect to $y$.
\begin{theorem}
There is not any left-invariant Randers metric of Douglas type on the Heisenberg group $H_{2n+1}$ which is Ricci-quadratic. 
\end{theorem}
\begin{proof}
It is worth mentioning that there is a bijective correspondence between left-invariant Randers metric of Douglas type on $H_{2n+1}$ and the set $\{U\in \mathbb{H}_{2n+1}| <U,U><1 ,\quad <U,Z>=0 \}$. Suppose that $Q=\sum_{i=1}^n a_i U_i+ b_i V_i$ is a left-invariant vector field such that
\[F(a,v)=\sqrt{g(v,v)}+g(Q(a),v),\]
is a left-invariant Randers-Douglas metric on $H_{2n+1}$, and $F$ is Ricci-quadratic. Then by (\cite{Deng}, Theorem 7.9) $F$ is Berwald type. But the relations in theorem \ref{levi} show that $Q=0$, which is a contradiction.
\end{proof}
We recall that a naturally reductive homogeneous space, is a reductive homogeneous Riemannian manifold $M=\dfrac{G}{H}$ with a decomposition $\mathfrak{g}=\mathfrak{h}+\mathfrak{m}$, that satisfies
\[ \langle [X,Y]_\mathfrak{m} ,Z\rangle =-\langle Y, [X,Z]_\mathfrak{m} \rangle, \qquad X,Y,Z\in \mathfrak{m}.
\]
Moreover, when $H=\{e\}$, then $\mathfrak{m}=\mathfrak{g}$ and the above condition can be rewrite as follows
\[ \langle [X,Y] ,Z\rangle +\langle Y, [X,Z] \rangle =0, \qquad X,Y,Z\in \mathfrak{g}.\]
Now we are able to prove the following result.
\begin{theorem}
Randers Heisenberg group $H_{2n+1}$ of dimension $2n+1$ is never naturally reductive.
\end{theorem}
\begin{proof}
Suppose, conversely, that a Randers metric $F$ on $H_{2n+1}$ is naturally reductive. Then by theorem 4.1 in \cite{latifi} the underling Riemannain metric is naturally reductive and we have
\[\langle [U,V] ,W\rangle +\langle V, [U,W] \rangle =0, \qquad U,V,W\in \mathbb{H}_{2n+1}.\]
If we replace $(U,V,W)$ by $(U_1, V_1 ,Z)$, then we have the contradiction $\lambda =0$.
\end{proof}
As we observed in theorem \ref{berwa}, the Heisenberg group $H_{2n+1}$ can not admit left-invariant Randers metric of Berwald type. A special family of non Berwald left-invariant Randers metrics which give us a geometric relationship between the Lie algebra $\mathbb{H}_{2n+1}$ and the Randers metrics are $Z$-Randers metrics. We say a left-invariant Randers metric 
\[F(a,v)=\sqrt{g(v,v)}+g(Q(a),v)\] 
on $H_{2n+1}$ is $Z$-Randers metric if and only if $Q\in \mathrm{span}<Z>$. In fact, the condition $Q\in \mathrm{span}<Z>$ guarantees that the Randers metric is not Berwald.\\

A homogeneous Finsler space $(M,F)$ is said to be a geodesic orbit space if every geodesic in $M$ is an orbit of 1-parameter group of isometries. More details on such spaces can be found in \cite{latif}. 
\begin{theorem}
All Douglas Hiesenberg groups $H_{2n+1}$ are not geodesic orbit spaces. The only Randers metrics on Hiesenberg group $H_{2n+1}$ which could be geodesic orbit are $Z$-Randers metric.
\end{theorem}
\begin{proof}
Suppose that $Q=\sum_{i=1}^n a_i U_i+ b_i V_i$ is a left-invariant vector field such that
\[F(a,v)=\sqrt{g(v,v)}+g(Q(a),v),\]
is a left-invariant Randers-Douglas metric on $H_{2n+1}$. If $(H_{2n+1},F)$ be a geodesic orbit Finsler space, then by Corollary 5.3 of \cite{latif}, the $S$-curvature vanishes which means $ad(Q)$ is a skew symmetric on the $\mathbb{H}_{2n+1}$. But we have
\[ [Q,U_i]=-b_i Z,\qquad [Q,V_i]=a_i Z,\qquad [Q,Z]=0.
\]
So, $ad(Q)$ is skew symmetric if and only if $a_i=b_i=0$ for $1\leq i\leq n$, which is a contradiction.\\
Now, let $F(a,v)=\sqrt{g(v,v)}+g(Q(a),v)$ is a left-invariant Randers metric on $H_{2n+1}$ with $Q=aZ+\sum_{i=1}^n a_i U_i+ b_i V_i$. If $(H_{2n+1},F)$ be a geodesic orbit Finsler space, then a similar argument applies to $ad(Q)$ derives the assertion. 
\end{proof}
Recall that a connected Finsler space $(M,F)$ is said to be a weakly symmetric space if for every two points $p$ and $q$ in $M$ there exists an isometry $\phi$ in the complete group of isometries $I(M,F)$ such that $\phi(p)=q$. A weakly symmetric Finsler space must be a geodesic orbit Finsler space by Theorem 6.3 in \cite{Deng}.
\begin{corollary}
All Douglas Hiesenberg groups $H_{2n+1}$ are never weakly symmetric spaces.
\end{corollary}
In the rest of this section, we will restrict our attention to the geometry of $Z$-Randers metrics on the Heisenberg group $H_{2n+1}$.
\begin{theorem}
For every $Z$-Randers metric on the Heisenberg group $H_{2n+1}$ the flag curvature of the Chern-Rund
connection with flag pole $W=\dfrac{1}{\sqrt{\lambda}}Z$ is given by the following equation.
\[
K(\Pi ,W)=K(\Pi)=\dfrac{\lambda}{4}\dfrac{\sum_{i=1}^n \dfrac{(a_i^2 +b_i^2)}{\sigma_i}}{\sum_{i=1}^n (a_i^2 +b_i^2)\sigma_i},
\]
where, $\Pi=span \{U,W\}$ and $U=\sum_{i=1}^n a_i U_i+b_iV_i$.
\end{theorem}
\begin{proof}
Suppose that $F(a,v)=\sqrt{g(v,v)}+g(Q(a),v)$ is a $Z$-Randers metric on $H_{2n+1}$ with $Q=\xi W=\dfrac{\xi}{\sqrt{\lambda}}Z$ for a real number $0<\xi <1$. Using method described in (\cite{Bert}, Theorem 3.10) and applied in \cite{lengy} we calculate the osculating metric and the Chern-Rund connection directly. One can easily check the following relations.
\begin{align*}
\langle U_i, U_j \rangle_W &= (1+\xi)\delta_{ij}\sigma_i ,\qquad
\langle V_i, V_j \rangle_W =(1+\xi)\delta_{ij}\sigma_i ,\\
\langle U_i, V_j \rangle_W &=0 ,\qquad \qquad\qquad
\langle U_i, W \rangle_W =0,\\
\langle V_i, W \rangle_W &=0 ,\qquad \qquad\qquad
\langle W, W \rangle_W =(1+\xi )^2.\\
\end{align*}
The local components of the Chern-Rund connection associated to the osculating metric $\langle .,.\rangle_W$ with respect to the basis $\{U_i ,V_i ,W\}_{i=1}^n$ which is denoted by $\overline{\nabla}$ are given by the following formulas.
\begin{align*}
&\overline{\nabla}_{U_i} V_i =\dfrac{\sqrt{\lambda}}{2}W,\qquad\qquad \overline{\nabla}_{V_i} U_i=-\dfrac{\sqrt{\lambda}}{2}W,\\
&\overline{\nabla}_{U_i} W =-\dfrac{\sqrt{\lambda}}{2\sigma_i}(1+\xi)V_i,\qquad\quad \overline{\nabla}_W {U_i} =-\dfrac{\sqrt{\lambda}}{2\sigma_i}(1+\xi)V_i,\\
&\overline{\nabla}_{V_i} W =\dfrac{\sqrt{\lambda}}{2\sigma_i}(1+\xi)U_i,\qquad\qquad \overline{\nabla}_W {V_i} =\dfrac{\sqrt{\lambda}}{2\sigma_i}(1+\xi)U_i.
\end{align*}
According to the  above relations, for the Riemannian curvature of the Chern-Rund connection $\overline{\nabla}$ denoted by $\overline{R}$, we have
\[
\overline{R}(U_i,W)(W)=\dfrac{\lambda}{4\sigma_i^2}(1+\xi)^2 U_i, \qquad \overline{R}(V_i,W)(W)=\dfrac{\lambda}{4\sigma_i^2}(1+\xi)^2 V_i.
\]
Applying the above computations, we find that
\begin{align*}
 K(\Pi ,W)&=\dfrac{\langle \overline{R}(U_i,W)W,U_i\rangle_W}{ \langle W,W\rangle_W \langle U_i,U_i\rangle_W -\langle W,U_i\rangle_W^2}=\dfrac{\lambda}{4\sigma_i^2},\\
 K(P ,W)&=\dfrac{\langle \overline{R}(V_i,W)W,V_i\rangle_W}{ \langle W,W\rangle_W \langle V_i,V_i\rangle_W -\langle W,V_i\rangle_W^2}=\dfrac{\lambda}{4\sigma_i^2},
\end{align*}
where, $\Pi=span \{U_i,W\}$ and $P=span \{V_i,W\}$. Now, let $U=\sum_{i=1}^n a_i U_i+b_iV_i$ and\newline $\Pi=span \{U,W\}$, then easy calculations show
\[K(\Pi ,W)=K(\Pi)=\dfrac{\lambda}{4}\dfrac{\sum_{i=1}^n \dfrac{(a_i^2 +b_i^2)}{\sigma_i}}{\sum_{i=1}^n (a_i^2 +b_i^2)\sigma_i}.
\]
\end{proof}
The above theorem shows that the flag curvature of every $Z$-Randers metric on the Heisenberg group $H_{2n+1}$ with flag pole $W=\dfrac{1}{\sqrt{\lambda}}Z$ is strictly positive. 
\begin{theorem}
For every $Z$-Randers metric on the Heisenberg group $H_{2n+1}$ the flag curvature of the Chern-Rund
connection with flag pole $W=U_1$ is given by the following equations.
\begin{align*}
 K(\Pi ,W)&=\dfrac{\xi^2 {\lambda}}{4\sigma_2},\\
 K(P ,W)&=\dfrac{\lambda^2 (1+\xi^2)^2}{4\sigma_1^2 \lambda (1+\xi^2(1-\sigma_1^2))},
\end{align*}
where, $\Pi=span \{U_2,W\}$ and $P=span \{Z,W\}$. 
\end{theorem}
\begin{proof}
Suppose that $F(a,v)=\sqrt{g(v,v)}+g(Q(a),v)$ is a $Z$-Randers metric on $H_{2n+1}$ with $Q=\dfrac{\xi}{\sqrt{\lambda}}Z$ for a real number $0<\xi <1$. If $W=U_1$ then,  one can easily check the following relations.
\begin{align*}
\langle U_i, U_j \rangle_W &=\delta_{ij} \sigma_i  ,\qquad
\langle V_i, V_j \rangle_W =\delta_{ij} \sigma_i ,\\
\langle U_i, V_j \rangle_W &=0 ,\qquad \quad
\langle U_i, Z \rangle_W =\xi\sqrt{\lambda} \delta_{1i}\sigma_i ,\\
\langle V_i, Z \rangle_W &=0 ,\qquad 
\langle Z, Z \rangle_W =\lambda (1+ \xi^2).\\
\end{align*}
The local components of the Chern-Rund connection associated to the osculating metric $\langle .,.\rangle_W$ with respect to the basis $\{U_i ,V_i ,Z\}_{i=1}^n$ which is denoted by $\overline{\nabla}$ are given by the following formulas.
\begin{align*}
&\overline{\nabla}_{U_i} V_j =\dfrac{1}{2}\delta_{ij}Z + \dfrac{\xi \sqrt{\lambda}}{2}\sum_{k=1}^n (\delta_{ij}\delta_{1k}+\delta_{1i}\delta_{jk}\dfrac{\sigma_i}{\sigma_k})U_k=\delta_{ij}Z+\overline{\nabla}_{V_j} U_i,\\
&\overline{\nabla}_{U_i} U_j=-\dfrac{\xi \sqrt{\lambda}}{2}\sum_{k=1}^n (\delta_{1i} \delta_{jk}\dfrac{\sigma_i}{\sigma_k} +\delta_{1j}\delta_{ik}\dfrac{\sigma_j}{\sigma_k})V_k =\overline{\nabla}_{U_j} {U_i},\\
&\overline{\nabla}_{U_i} Z=-\dfrac{\lambda}{2\sigma_i}(1+\xi^2)V_i=\overline{\nabla}_{Z} U_i,\\
&\overline{\nabla}_{V_i} Z=\dfrac{\lambda}{2\sigma_i}(1+\xi^2)U_i=\overline{\nabla}_{Z}V_i,
\end{align*}
According to the  above relations, for the Riemannian curvature of the Chern-Rund connection $\overline{\nabla}$ denoted by $\overline{R}$, we have
\begin{align*}
&\overline{R}(U_2,W)(W)=\dfrac{\xi^2 {\lambda}}{4}\dfrac{\sigma_1}{\sigma_2} U_2, \\
&\overline{R}(Z,W)(W)=\dfrac{\lambda}{4\sigma_1}(1+\xi^2)Z.
\end{align*}
Applying the above computations, we find that
\begin{align*}
 K(\Pi ,W)&=\dfrac{\langle \overline{R}(U_2,W)W,U_2\rangle_W}{ \langle W,W\rangle_W \langle U_2,U_2\rangle_W -\langle W,U_2\rangle_W^2}=\dfrac{\xi^2 {\lambda}}{4\sigma_2},\\
 K(P ,W)&=\dfrac{\langle \overline{R}(Z,W)W,Z\rangle_W}{ \langle W,W\rangle_W \langle Z,Z\rangle_W -\langle W,Z\rangle_W^2}=\dfrac{\lambda^2 (1+\xi^2)^2}{4\sigma_1^2 \lambda (1+\xi^2(1-\sigma_1^2))},
\end{align*}
where, $\Pi=span \{U_2,W\}$ and $P=span \{Z,W\}$. 
\end{proof}
\begin{remark}
Considering some special cases of $\xi$ and $\sigma_1$, the above theorem shows that there exist flags of strictly negative and strictly positive curvatures on Heisenberg groups.
\end{remark}
%%%%%%%%%%%%%%%%%%%%%%%%%%%%%%%%%%%
\section{Conclusion}

In this paper, we investigated the geometry of left-invariant Randers metrics on the Heisenberg group $H_{2n+1}$, of dimension $2n+1$. We determined the Levi-Civita connection, curvature tensor, Ricci tensor and scalar curvature of a left-invariant metric on  $H_{2n+1}$ and showed  the Heisenberg groups $H_{2n+1}$  have constant negative scalar curvature. Other geometric properties of such spaces are investigated. Also, we showed the Heisenberg group $H_{2n+1}$ can not admit Randers metric of Berwald and Ricci-quadratic Douglas types. Finally, by computing flag curvature it was shown that there exist flags of strictly negative and strictly positive curvatures. \\
The question of whether the Heisenberg group $H_{2n+1}$ admits a Randers metric of general Dauglus type, is not considered in this paper. Also, we have computed the flag curvature of a special kind of Randers metric (namely $Z$-Randers metric) on $H_{2n+1}$ and giving an explicit formula for computing flag curvature in general case would be a matter of another paper.  
%%%%%%%%%%%%%%%%%%%%%%%%%%%%%%%%%%%%%%%%%%%%%%%
%%%%%%%%%%%%%%%%%%%%%%%%%%%%%%%%%%%%%%%%%%%%%%%

\end{document}